\documentclass[12pt]{amsart}

\usepackage{amsthm, amsmath, amssymb}
\usepackage{enumerate}
\usepackage{microtype}

\usepackage{color}
\definecolor{darkgreen}{rgb}{0,0.35,0}
\definecolor{darkblue}{rgb}{0,0,0.6}
\usepackage[breaklinks,pdftex,colorlinks,citecolor=darkgreen,linkcolor=darkblue]{hyperref}

\newtheorem{theorem}{Theorem}[section]
\newtheorem{claim}{Claim}[theorem]

\newtheorem{lemma}[theorem]{Lemma}

\newtheorem{conjecture}[theorem]{Conjecture}

\DeclareMathOperator{\GF}{GF}
\DeclareMathOperator{\Pf}{Pf}
\DeclareMathOperator{\mindeg}{mindeg}
\DeclareMathOperator{\Z}{\mathbb{Z}}

\DeclareMathOperator{\bR}{\mathbb{R}}
\DeclareMathOperator{\rank}{rank}
\DeclareMathOperator{\odd}{odd}
\DeclareMathOperator{\wt}{wt}
\renewcommand{\O}{\mathcal{O}}
\newcommand{\cC}{\mathcal{C}}
\newcommand{\cP}{\mathcal{P}}
\newcommand{\cM}{\mathcal{M}}

\newcommand{\del}{\setminus}
\newcommand{\con}{/}

\title[Computing girth and cogirth]{Computing girth and cogirth in perturbed graphic matroids}

\author{Jim Geelen}
\address{Department of Combinatorics and Optimization, University of Waterloo, Waterloo, Ontario, Canada} 
\email{jim.geelen@uwaterloo.ca}

\author{Rohan Kapadia}
\address{Department of Computer Science and Software Engineering, Concordia University, Montr\'eal, Qu\'ebec, Canada}
\email{rohan.f.kapadia@gmail.com}

\thanks{This research was partially supported by a grant from the Office of 
Naval Research [N00014-10-1-0851]. }

\subjclass[2010]{05B35, 94B05, 90C27}
\keywords{matroid theory, coding theory, distance, girth, randomized algorithm}
\date{April 28, 2015}

\begin{document}

\begin{abstract}
We give polynomial-time randomized algorithms for computing the
girth and the cogirth of binary matroids that are low-rank
perturbations of graphic matroids.
\end{abstract}

\maketitle

\section{Introduction}
The {\em girth} of a matroid is the length of its shortest circuit;
if the matroid has no circuit, the girth is defined to be $\infty$.
The following two theorems are our main results:

\begin{theorem} \label{thm:main1}
Let $t$ be a positive integer and let $\epsilon>0$.
There is a randomized algorithm that, given matrices $A,P\in \GF(2)^{r\times n}$ where 
$A$ is the incidence matrix of a graph and $\rank(P) \le t$, will,
with probability at least $1-\epsilon$, 
correctly compute the girth of $M(A+P)$ in time $\O(r^7  \log^2 r + nr)$.
\end{theorem}

The {\em cogirth} of a matroid is the girth of its dual.

\begin{theorem} \label{thm:main2}
Let $t$ be a positive integer and let $\epsilon>0$.
There is a randomized algorithm that, given matrices $A,P\in \GF(2)^{r\times n}$ where 
$A$ is the incidence matrix of a graph and $\rank(P) \le t$, will,
with probability at least $1-\epsilon$,
correctly compute the cogirth of $M(A+P)$ in time $\O(r^5 n)$.
\end{theorem}

\subsection*{Cycles and cocycles}
Let $A\in\GF(2)^{r\times E}$. A {\em cycle} of $M(A)$ is a subset $C$
of $E$ such that the columns of $A$ indexed by $C$ sum to zero.
Thus $C$ is a cycle if and only if it is a disjoint union of circuits.
The girth of $M(A)$ is the size of the smallest non-empty cycle;
this turns out, for the purpose of this paper, to be the 
most convenient way to view girth.

A {\em cocycle} of $M(A)$ is a set whose characteristic vector
is in the row-span of $A$. Equivalently, $C^*$ is a cocycle of
$M(A)$ if and only if it is a cycle of $M(A)^*$.
So the cogirth of $M(A)$ is the size of the smallest
non-empty cocycle. Again, for this paper, this is
the most convenient way to view cogirth.

\subsection*{Motivation}
The problem of computing the girth of a binary matroid has received a lot of attention
due to its well-known connection with coding theory. If $A$ is the parity-check matrix of
a binary linear code $\cC$, then the distance of $\cC$ is equal to the 
girth of the binary matroid $M(A)$. In a landmark paper, Vardy \cite{Vardy} proved that 
the problem of computing girth in binary matroids is $\mathcal{NP}$-hard.
On the other hand, there are significant classes of binary matroids in which
one can efficiently compute girth; for example, the class of graphic matroids and the class
of cographic matroids.  Geelen, Gerards, and Whittle \cite{GeelenGerardsWhittle} posed the following
conjecture.

\begin{conjecture}\label{conj:girth}
For any proper minor-closed class $\cM$ of binary matroids, there is a polynomial-time
algorithm for computing the girth of matroids in $\cM$.
\end{conjecture}
 
Here a minor-closed class of binary matroids is called {\em proper} if it does not, up to isomorphism,
contain all binary matroids.
In the same paper, Geelen, Gerards, and Whittle announced (without proof) the following result:
\begin{theorem}\label{thm:structure}
For each proper minor-closed class $\cM$ of binary matroids, there exist non-negative integers
$k$ and $t$ such that, for each vertically $k$-connected matroid $M\in \cM$,  there exist 
matrices $A,P\in\GF(2)^{r\times n}$ such that $A$ is the incidence matrix of a graph, 
$\rank(P)\le t$, and either $M=M(A+P)$ or $M=M(A+P)^*$.
\end{theorem}

In light of this result, our Theorems~\ref{thm:main1} and~\ref{thm:main2}
give significant support to Conjecture~\ref{conj:girth}; their main shortcomings being:
(1) they only apply to sufficiently connected matroids in a minor-closed class, and
(2) they only give {\em randomized} algorithms.

\subsection*{Related work}
Barahona and Conforti \cite{BarahonaConforti} studied both the
girth and cogirth problems for the class of ``even-cycle matroids".
Let $M_1$ and $M_2$ be binary matroids on the same ground set.
We call $M_1$ a {\em rank-$t$ perturbation} of $M_2$ if $M_1$ 
has a representation $A$ and $M_2$ has a representation $B$
such that $B - A$ has rank $t$.
We call $M_1$ an {\em even-cycle matroid} if it is a rank-$1$
perturbation of a graphic matroid $M_2$ and $r(M_1)=r(M_2)+1$.

Barahona and Conforti 
gave an efficient deterministic algorithm for computing 
girth of even-cycle matroids. They also noted that the problem
of computing cogirth for these matroids is closely related to
the max-cut problem; however, they neither found an efficient
algorithm for the cogirth problem nor proved that it is NP-hard.

\subsection*{This paper}
Perturbations of graphic matroids can be encoded as labelled
graphs; see Lemma~\ref{lem:decomposition}.
Using this result we reduce Theorem~\ref{thm:main1}
to the $t$-Dimensional Parity Join Problem, discussed
in Section~6, which we reduce, in turn, to the
$t$-Dimensional Parity Perfect Matching Problem,
discussed in Section~5. We solve the $t$-Dimensional
Parity Perfect Matching Problem  using a variant of the
Mulmuly, Vazarani, and Vazarani algorithm~\cite{MVV}
for the exact matching problem.

We then employ Lemma~\ref{lem:decomposition} again
to reduce Theorem~\ref{thm:main2} to the $t$-Dimensional
Even Cut Problem, which we solve, in Section~3,
by a variant of Karger's algorithm~\cite{Karger} for the 
global minimum cut problem.

It is curious that, while our two algorithms are quite different from each other, they  both require randomization.  Finding efficient {\em deterministic}
algorithms seems to be quite difficult; below we discuss two particular bottlenecks.

\subsection*{Even-cut problem}
Let $k$ be a fixed non-negative integer.
An instance of the {\em $k$-set even-cut problem} consists of a triple $(G;T_1,\ldots,T_k)$
where $T_1,\ldots,T_k$ are even-cardinality subsets of $V(G)$.
The problem is, among all non-empty proper subsets $X$ of $V(G)$ with
$|T_1\cap X|,\ldots, |T_k\cap X|$ all even, to minimize
the size of the cut $\delta_G(X)$. Here $\delta_G(X)$ denotes the set of
all edges of $G$ that have one end in $X$ and one end in $V(G)-X$.

We give a polynomial-time {\em randomized} algorithm for the $k$-set even-cut problem; see Section~\ref{sec:evencuts}.
Conforti and Rao \cite{ConfortiRao} found an efficient {\em deterministic}
algorithm for the one-set even-cut problem, but we have not been able to
find a deterministic solution for the two-set version.

\subsection*{A parity matching problem}
An instance of the {\em weighted even perfect matching problem}  consists of 
a triple $(G,\Sigma,w)$ where $G$ is a graph, $\Sigma\subseteq E(G)$, and
$w:E(G)\rightarrow\{0,1\}$ is an edge weighting.
The problem is, among all perfect matchings $M$ of $G$ with $|M\cap \Sigma|$ even,
to minimize $\sum (w(e)\, : \, e\in M)$.

Anyone who is familiar with the Mulmuly, Vazarani, and Vazarani~\cite{MVV} matrix formulation of the
exact matching problem (see~\cite{MVV}) will recognize that the 
weighted even perfect matching problem
can be solved by an efficient randomized algorithm.
However, it is not even clear how one might solve the 
feasibility problem deterministically.

The relationship between the two problems above 
and the problem of computing girth will become clear.

\section{Even cuts} \label{sec:evencuts}

In this section we give an efficient randomized algorithm for the $t$-Set Even-Cut Problem.
This section is peripheral to the rest of paper and may freely be skipped by the reader.
We include the material since we believe that this problem is of independent interest.

\medskip

{\it
\noindent
{\bf The $t$-Set Even-Cut Problem}\\[1mm]
{\sc Instance:}
A tuple $(G;T_1,\ldots,T_t)$
where $G$ is a graph and $T_1,\ldots,T_t$ are 
even-cardinality subsets of $V(G)$.\\
{\sc Problem:}
Among all non-empty proper subsets $X$ of $V(G)$ with
$|T_1\cap X|,\ldots,|T_t\cap X|$ all even, 
minimize the size of the cut $\delta_G(X)$.
}
\medskip

Let $(G; T_1,\ldots,T_t)$ be an instance of the $t$-Set Even-Cut Problem.
A {\em $(T_1,\ldots,T_t)$-even cut} is a cut $\delta(X)$ such that $\emptyset\subset X\subset V(G)$ and each of
$|T_1\cap X|,\ldots,|T_t\cap X|$ is even.
Note that $(T_1,\ldots,T_t)$-even cuts do not always exist; for example,
the problem is infeasible when $|V(G)|=2$ and $T_1=V(G)$.
However, it is easy to check feasibility.
To see this, consider the matrix $A \in\GF(2)^{t\times V(G)}$ where the
$i$-th row of $A$ is the characteristic vector of the set $T_i$.
For $X\subseteq V(G)$, the cut $\delta(X)$ is $(T_1,\ldots,T_k)$-even if and only
if $X$ is a cycle of $M(A)$ and $\emptyset\subset X\subset V(G)$.
Thus $(G;T_1,\ldots,T_t)$ is feasible unless 
$V(G)$ is a circuit of $M(A)$ (note that, since $|T_1|,\ldots,|T_t|$ are even $V(G)$ is itself a 
cycle of $M(A)$). In particular, if $|V(G)|\ge t+2$, then the instance is feasible.

The following is a randomized algorithm for solving the $t$-Set Even-Cut Problem.
The algorithm, as well as the analysis that follows, is based on a randomized algorithm for finding minimum
cuts due to Karger~\cite{Karger}. First we need some notation.

Let $e$ be an edge with ends $x$ and $y$ in $G$, and let $G\con e$ denote the graph obtained
by contracting $e$ to a new vertex $z$.  For each $i\in \{1,\ldots,t\}$, there is a unique even-cardinality 
subset $T'_i$ of $G\con e$ such that $T_i-\{x,y\} = T'_i-\{z\}$. We denote the tuple $(G\con e; T'_1,\ldots,T'_t)$ by 
$(G;T_1,\ldots,T_t)\con e$. 

\medskip
{\it
\noindent
{\bf Random Contraction Algorithm}\\[1mm]
{\sc Input:} A feasible instance $(G;T_1,\ldots,T_t)$ of the $t$-Set Even-Cut Problem.\\[1mm]
{\sc Step 1.} Delete the loops of $G$.\\
{\sc Step 2.} If $|V(G)|\le 2^t + 4$, find a minimum cardinality $(T_1,\ldots,T_t)$-even cut $C$ by 
 exhaustive search and then stop and return $C$.\\
{\sc Step 3.} If $G$ has no edge stop and return $\emptyset$.\\
{\sc Step 4.} Choose an edge $e$ of $G$ uniformly at random and replace the instance
 $(G;T_1,\ldots,T_t)$ with $(G;T_1,\ldots,T_t)\con e$. Then repeat from Step~1.
}
\medskip

For an $n$-vertex $m$-edge graph, the Random Contraction Algorithm takes $\O(nm)$ time (actually Karger's algorithm can be executed even faster than this; see~\cite{KargerStein}).

\begin{lemma} \label{lem:claim1}
 Let $(G; T_1, \ldots, T_t)$ be a feasible instance of the $t$-Set Even-Cut Problem and let $k$ be the minimum size of
 a $(T_1,\ldots,T_t)$-even cut in $G$.
 Then $(|V(G)| - 2^{t}) k \leq 4 |E(G)|$.
\end{lemma}

\begin{proof}
Define an equivalence relation $(V(G),\sim)$ where $u\sim v$ if and only if, for each $i\in\{1,\ldots,t\}$,
the set $T_i$ contains either none or both of $u$ and $v$. Now let $\Pi$ be the partition of 
$V(G)$ into equivalence classes; note that $|\Pi|\le 2^t$.
 For each set $P\in \Pi$, fix an ordering $(v_1,\ldots,v_{|P|})$ of the elements of
 $P$  and let $X_P = \{\{v_i,v_{i+1}\}\, : \, 1\le i<|P|\}$.
 Note that, for each $X\in X_P$, the cut $\delta(X)$ is 
 $(T_1,\ldots , T_t)$-even.
 Moreover, each vertex of $P$ appears in at most two of the sets in $X_P$, so each edge of $G$ appears in at most four of cuts
 $(\delta(X)\, : \, P\in\Pi,\, X\in X_P)$.
 Therefore, $4 |E(G)| \geq \sum(|\delta(X)|\, : \, P\in\Pi,\, X\in X_P) \geq (|V(G)| - 2^t)k$.
\end{proof}

\begin{lemma} \label{lem:paritycuts1}
Let $(G; T_1, \ldots, T_t)$ be a feasible instance of the $t$-Set Even-Cut Problem and let $k$ be the
minimum size of a $(T_1,\ldots,T_t)$-even cut. Then the Random Contraction Algorithm returns a cut of size
$k$ with probability at least $\frac{24}{|V(G)|^4}$.
 \end{lemma}

\begin{proof}
Let $C^*$ be a minimum cardinality $(T_1,\ldots,T_t)$-even cut and let $n=|V(G)|$.
Consider an edge $e$ chosen in Step 4. Note that, if  $e\not\in C^*$, then $C^*$ remains optimal
for the instance $(G;T_1,\ldots,T_t)\con e$.
By Lemma~\ref{lem:claim1},
$$
 {\bf Prob} [e \notin C^*] = 1 - \frac{k}{|E(G)|} \geq \frac{n - 2^{t} - 4}{n - 2^{t}}.
$$

We repeat Step~4 a total of $n- 2^t - 4$ times on successively smaller graphs; the probability that 
we never choose an edge of $C^*$ is at least:
\begin{eqnarray*}
\frac{n - 2^{t} - 4}{n - 2^{t}} \cdot \frac{n - 2^{t} - 5}{n - 2^{t} - 1} \cdots \frac{1}{5} &= &\frac{4!}{(n - 2^{t})\cdots(n - 2^{t} - 3)}\\
&>& \frac{24}{n^4},
\end{eqnarray*}
as required.
\end{proof}

The bound $\frac{24}{|V(G)|^4}$ may not be that impressive, but this can be improved
through repetition. Observe that, if we apply the Random Contraction Algorithm to a feasible instance
$(G;T_1,\ldots,T_t)$, then the algorithm returns an even $(T_1,\ldots,T_t)$-cut.
We can repeatedly apply the algorithm, keeping the smallest of these cuts, to reduce the
error-probability.

\begin{theorem} \label{thm:paritycuts}
Let $t$ and $c$ be positive integers. 
Let $(G; T_1, \ldots, T_t)$ be a feasible instance of the $t$-Set Even-Cut Problem and let $k$ be the
minimum size of a $(T_1,\ldots,T_t)$-even cut. Then in $c|V(G)|^4$ repetitions of the Random Contraction Algorithm,
the probability that we fail to find a 
$(T_1,\ldots,T_t)$-even cut of size $k$ is at most $e^{-24c}$.
\end{theorem}

\begin{proof}
Let $n= |V(G)|$.
By Lemma~\ref{lem:paritycuts1}, the error-probability is at most
 \[ \left(1 - \frac{24}{n^4}\right)^{cn^4} \leq \left(e^{-\frac{24}{n^4}}\right)^{cn^4} = e^{-24c}. \qedhere\]
\end{proof}

Given that we can solve the $t$-Set Even-Cut Problem efficiently (with randomization),
it is natural to consider the following variation.

\medskip

{\it
\noindent
{\bf The $t$-Set Odd-Cut Problem}\\[1mm]
{\sc Instance:}
A tuple $(G;T_1,\ldots,T_t)$
where $G$ is a graph and $T_1,\ldots,T_t$ are 
even-cardinality subsets of $V(G)$.\\
{\sc Problem:}
Among all proper subsets $X$ of $V(G)$ with
$|T_1\cap X|,\ldots,|T_t\cap X|$ all odd, 
minimize the size of the cut $\delta_G(X)$.
}
\medskip

Padberg and Rao~\cite{pr} give a polynomial-time algorithm 
for the $1$-Set Odd-Cut Problem, and the same method extends easily
to the $2$-Set Odd-Cut Problem, but the complexity of the $3$-Set Odd-Cut 
Problem remains open. 

\section{A variation on even cuts}

To solve the cogirth problem on perturbed graphic matroids we will reduce it to a  variation on the
$t$-Set Even Cut Problem; in this section we will solve that variant.
The methods in this section are similar to those in the previous section, but
we will use different notation.

Let $G$ be a graph and $\tau : V(G)\rightarrow \GF(2)^{t}$.
We denote $\sum( \tau(v)\, : \, v\in X)$ by $\tau(X)$.

\medskip

{\it
\noindent
{\bf The $t$-Dimensional Even-Cut Problem}\\[1mm]
{\sc Instance:}
A tuple $(G,\tau, \Sigma, \alpha)$
where $G$ is a graph, $\tau : V(G)\rightarrow \GF(2)^{t}$,
$\Sigma\subseteq E(G)$, and $\alpha\in\GF(2)^t.$\\
{\sc Problem:}
Find the minimum size of a non-empty set $\delta_G(X)\Delta \Sigma'$
where $(\Sigma',\alpha') \in \{(\Sigma,\alpha), (\emptyset,0)\}$ and $X\subseteq V(G)$ with
$\tau(X)=\alpha'$.
}

\medskip

Consider an instance $(G;T_1,\ldots,T_t)$ of the $t$-Set Even-Cut Problem
where {\em $G$ is connected}. For each $v\in V(G)$ we let 
$\tau(v)\in \GF(2)^t$ where, for each $i\in\{1,\ldots,t\}$, we let $\tau(v)_i=1$ if
$v\in T_i$. This reduces our instance to an instance $(G,\tau,\emptyset,0)$ of 
the $t$-Dimensional Even-Cut Problem. When $G$ is not connected
the problems are not related, since, for the $t$-Dimensional Even-Cut Problem
we explicitly require a non-empty set as a solution. This difference between the 
problems adds a layer of difficulty. Another key difference between
the problems is that, for the $t$-Set Even-Cut Problem, we require the sets
$T_1,\ldots, T_t$ to be even, but for the $t$-Dimensional Even-Cut Problem
we do not require $\tau(V(G))=0$. Thus, for a set $X\subseteq V(G)$, it may not be
the case that $\tau(X)= \tau(V(G)-X)$. This lack of symmetry seems 
a little unnatural, but it does not cause any additional difficulty. A final
difference between the problems is the role of $\Sigma$, which does not 
add to the difficulty at all.

\subsection*{Cogirth}
We start by drawing a connection between this problem and the problem
of computing the cogirth of a binary matroid.
Consider an instance $(G,\tau, \Sigma,\alpha)$
of the $t$-Dimensional Even-Cut Problem.
Let $T=\{1,\ldots,t\}$ and let $A(G)$ be the incidence matrix of $G$.
Now let $B\in \GF(2)^{V(G)\times T}$ where the 
row of $B$ indexed by $v\in V(G)$ is $\tau(v)$, let
$\sigma\in\GF(2)^{E(G)}$ be the characteristic vector of $\Sigma$, and let
$$
A = \bordermatrix{ & E(G) & T \cr  & \sigma & \alpha \cr V(G) & A(G) & B}.
$$
We call $A$ the {\em incidence matrix} of $(G,\tau, \Sigma, \alpha)$, 
and we denote $M(A)\con T$ by $M(G,\tau, \Sigma,\alpha)$.
The next lemma is an easy consequence of these definitions.
\begin{lemma}\label{formulation1}
Let $(G,\tau, \Sigma,\alpha)$ be an instance of the
$t$-Dimensional Even-Cut Problem.
Then a set $C\subseteq E(G)$ is a cocycle of 
$M(G,\tau, \Sigma,\alpha)$ if and only if there exists
 $(\Sigma',\alpha')\in \{(\Sigma,\alpha), (\emptyset,0)\}$ and  $X\subseteq V(G)$ with
$\tau(X)=\alpha'$ such that $C = \delta_G(X)\Delta\Sigma'$.
\end{lemma}

\begin{proof} A set $C\subseteq E(G)$ is a cocycle of $M(A)\con T$ if and only if 
it is a cocycle of $M(A)$. Therefore the set of cocycles of $M(A)\con T$ consists
of the set of all sets obtained by taking the support of a vector $x$ in the rowspace of 
$A$ with $x|T=0$. Consider $\sigma'\in \GF(2)^{E(G)}$ and
$\alpha'\in \GF(2)^T$, and let $\Sigma'$ be the support of $\sigma'$.
Note that $(\sigma',\alpha')$ is in the row space of $(A(G),\, B)$ if and only if there 
exists $X\subseteq V(G)$ such that $\delta_G(X) =\Sigma'$ and $\tau(X)=\alpha'$.
Now the result follows easily.
\end{proof} 

So the $t$-Dimensional Even-Cut Problem is simply the problem 
of determining the cogirth of $M(G,\tau, \Sigma,\alpha)$. 

\subsection*{Connectivity reductions}
A {\em connected} instance  of the $t$-Dimensional Even-Cut Problem
is an instance $(G,\tau, \Sigma,\alpha)$ such that
$G$ is connected.
We will describe two reductions that, together, reduce
an instance $(G,\tau, \Sigma,\alpha)$ of the
$t$-Dimensional Even-Cut Problem to connected instances.

We call a component of a graph {\em trivial} if it has exactly one vertex.
The first reduction reduces us to an instance $(G',\tau',\Sigma,\alpha)$
in which $G'$ has at most one non-trivial component.
Let $H_1,\ldots,H_c$ denote the components of $G$ and,
for each $i\in\{1,\ldots,c\}$, choose a vertex $v_i$ of $H_i$.
Let  $G'$ be the graph obtained from $G$ by identifying
the set of vertices $\{v_1,\ldots,v_c\}$ to a single vertex $v_1$
and then adding new isolated vertices $v_2,\ldots,v_c$.
Now for each $v\in V(G')$ we define
$$
\tau'(v) = \left\{
\begin{array}{ll}
\tau(v), & v\not\in\{v_1,\ldots,v_c\} \\
\tau(v_1)+\cdots + \tau(v_c), & v=v_1 \\
\sum(\tau(w)\, : \, w\in V(H_i)), & v=v_i
\end{array}
\right.
$$

\begin{lemma}\label{lem:conred1}
$M(G,\tau, \Sigma,\alpha)= M(G',\tau', \Sigma,\alpha)$.
\end{lemma}

\begin{proof}
Let $A$ be the incidence matrix of $(G,\tau, \Sigma,\alpha)$ and
let $A'$ be the incidence matrix of $(G',\tau', \Sigma,\alpha)$.
Note that $A'$ is obtained from $A$ by a sequence of elementary row operations.
Thus $M(A) =M(A')$ and, hence, $M(G,\tau, \Sigma,\alpha) = M(G',\tau', \Sigma,\alpha)$.
\end{proof}

We may assume that $G'$ has a non-trivial component, say $G''$,
since otherwise we can easily compute the cogirth of
$M(G',\tau', \Sigma,\alpha)$. Let $\tau''$ be the restriction
of $\tau'$ to $V(G'')$.
\begin{lemma}\label{lem:conred2}
The cogirth of $M(G',\tau', \Sigma,\alpha)$ is equal to the minimum 
of the cogirths of $M(G'',\tau'', \Sigma,\alpha+\beta)$ taken over 
all $\beta$ in the span of $(\tau(v)\, : \, v\in V(G')-V(G''))$.
\end{lemma}

\begin{proof}
Let $A'$ be the incidence matrix of $(G',\tau', \Sigma,\alpha)$ and, for a vector
$\beta$ in the span of $(\tau(v)\, : \, v\in V(G')-V(G''))$, let
$A_{\beta}$ be the incidence matrix of $(G'',\tau'', \Sigma,\alpha+\beta)$.
It is easy to see that:
\begin{itemize}
\item each cocycle of $M(A_{\beta})$ is a cocycle of $M(A')$, and 
\item for each cocycle $C$ of $M(A')$ there is a vector
 $\beta$ in the span of $(\tau(v)\, : \, v\in V(G')-V(G''))$
 such that $C$ is a cocycle of $M(A_{\beta})$.
 \end{itemize}
Hence the result follows easily.
\end{proof}

Lemmas~\ref{lem:conred1} and~\ref{lem:conred2}
give a linear time reduction of an instance of the
$t$-Dimensional Even-Cut Problem to at most $2^t$
connected instances.

\subsection*{Feasibility} By Lemma~\ref{formulation1},
an instance $(G,\tau, \Sigma,\alpha)$ of the 
$t$-Dimensional Even-Cut Problem is feasible if and only if
$M(G,\tau, \Sigma,\alpha)$ has positive rank; the following 
result gives a simple sufficient condition.
\begin{lemma}
Let $(G,\tau, \Sigma,\alpha)$ be a connected instance of the 
$t$-Dimensional Even-Cut Problem. If $|V(G)|\ge t+2$,
then $(G,\tau, \Sigma,\alpha)$ is feasible.
\end{lemma}

\begin{proof} Let $A$ be the incidence matrix of 
$(G,\tau, \Sigma,\alpha)$. Now,
if $|V(G)|\ge t+2$, then
$$ \rank(A) \ge |V(G)|-1 > t. $$
Therefore $M(A)\con T$ has positive rank, and, hence,
$M(G,\tau, \Sigma,\alpha)$ has a non-empty cocycle.
\end{proof}

\subsection*{The algorithm}
Consider an instance $(G,\tau, \Sigma,\alpha)$ of the $t$-Dimensional Even-Cut Problem.
Let $A$ be the incidence matrix of $(G,\tau, \Sigma,\alpha)$  and let $e$ be an edge with ends $x$ and $y$ in $G$.
We will describe a new instance $(G',\tau', \Sigma',\alpha')$, with incidence matrix $A'$,
such that $M(A)\con e = M(A')$.
Let $G\con e$ denote the graph obtained
by contracting $e$ to a new vertex $z$.  For each $v\in V(G')$, we define
$$
\tau'(v) = \left\{
\begin{array}{ll} \tau(x)+\tau(y), & v=z \\ \tau(v), &\mbox{otherwise.}
\end{array}\right.
$$
If $e\not\in\Sigma$, we let $\Sigma'=\Sigma$ and $\alpha'=\alpha$.
If $e\in \Sigma$, we let $\Sigma'=\Sigma\Delta \delta_G(x)$ and $\alpha'=\alpha+\tau(x)$.
We denote $(G',\tau',\Sigma',\alpha')$ by $(G,\tau,\Sigma,\alpha)\con e$; note that
there is some ambiguity here since $x$ plays a distinguished role, but 
it does not matter which end of $e$ we choose for the algorithm. 
The following lemma is easy, we omit the proof.
\begin{lemma}\label{lem:contract}
Let $A$ be the incidence matrix of $(G,\tau, \Sigma,\alpha)$ and let 
$A'$ be the incidence matrix of $(G,\tau, \Sigma,\alpha)\con e$. Then
$M(A)\con e = M(A')$.
\end{lemma}
With this notation in place, we can state our algorithm.

\medskip
{\it
\noindent
{\bf Random Contraction Algorithm (revised)}\\[1mm]
{\sc Input:} A feasible  connected instance $(G,\tau, \Sigma,\alpha)$ of the $t$-Dimensional Even-Cut Problem.\\[1mm]
{\sc Step 1.} If $|V(G)|\le 2^t+4$, find a minimum cardinality cocycle $C$ 
of $M(G,\tau, \Sigma,\alpha)$ by 
 exhaustive search and then stop and return $C$.\\
{\sc Step 2.} Choose a non-loop edge $e$ of $G$ uniformly at random and replace the instance
 $(G,\tau, \Sigma,\alpha)$ with $(G,\tau, \Sigma,\alpha)\con e$. Then repeat from Step~1.
}
\medskip

For an $n$-vertex $m$-edge graph, the Random Contraction Algorithm takes $\O(nm)$ time 
(note that Step~1 can be done efficiently because $M(G,\tau, \Sigma,\alpha)$ has rank at most $|V(G)| + 1$).
The following analysis is the same as for the $t$-Set Even-Cut Problem. 

\begin{lemma} \label{lem:claim1v2}
 Let $(G,\tau, \Sigma,\alpha)$ be a feasible connected instance of the $t$-Dimensional Even-Cut Problem and let $k$ be the 
 optimal value. If $\ell$ is the number of loops in $G$, then $(|V(G)| - 2^{t}) k \leq 4 (|E(G)| - \ell)$.
\end{lemma}

\begin{proof}
Let $\Pi$ be the partition of $V(G)$ into sets with equal $\tau$-value; thus $|\Pi|\le 2^t$.
 For each set $P\in \Pi$, fix an ordering $(v_1,\ldots,v_{|P|})$ of the elements of
 $P$  and let $X_P = \{\{v_i,v_{i+1}\}\, : \, 1\le i<|P|\}$.
 Note that, for each $X\in X_P$, the cut $\delta(X)$ is 
 a cocycle of $M(G,\tau, \Sigma,\alpha)$, and, hence, $|\delta(X)| \ge k$.
 Moreover, each vertex of $P$ appears in at most two of the sets in $X_P$, so each edge of $G$ appears in at most four of 
the cocycles $(\delta(X)\, : \, P\in\Pi,\, X\in X_P)$.
 Therefore, $4 (|E(G)| - \ell) \geq \sum(|\delta(X)|\, : \, P\in\Pi,\, X\in X_P) \geq (|V(G)| - 2^{t})k$.
\end{proof}

\begin{lemma} \label{lem:paritycuts1v2}
Let $(G,\tau, \Sigma,\alpha)$ be a feasible connected instance of the $t$-Dimensional Even-Cut Problem and let $k$ be the
optimal value. Then the Random Contraction Algorithm returns a cocycle of size
$k$ with probability at least $\frac{24}{|V(G)|^4}$.
 \end{lemma}

\begin{proof}
Let $C^*$ be a minimum cardinality non-empty cocycle of $M(G,\tau, \Sigma,\alpha)$ and let $n=|V(G)|$.
Consider an edge $e$ chosen in Step~2. Note that, if  $e\not\in C^*$, then $C^*$ remains optimal
for the instance $(G,\tau, \Sigma,\alpha)\con e$.
Let $\ell$ be the number of loops in $G$. By Lemma~\ref{lem:claim1v2},
$$
 {\bf Prob} [e \notin C^*] \geq 1 - \frac{k}{|E(G)| - \ell} \geq \frac{n - 2^{t} - 4}{n - 2^{t}}.
$$

We repeat Step~2 a total of $n- 2^t - 4$ times on successively smaller graphs; the probability that 
we never choose an edge of $C^*$ is at least:
\begin{eqnarray*}
\frac{n - 2^{t} -4}{n - 2^{t}} \cdot \frac{n - 2^{t} - 5}{n - 2^{t} - 1} \cdots \frac{1}{5} &= &\frac{4!}{(n - 2^{t})\cdots(n - 2^{t} - 3)}\\
&>& \frac{24}{n^4},
\end{eqnarray*}
as required.
\end{proof}

Observe that, if we apply the Random Contraction Algorithm to a feasible instance
$(G,\tau, \Sigma,\alpha)$, then the algorithm returns a cocycle of $M(G,\tau, \Sigma,\alpha)$.
We can repeatedly apply the algorithm, keeping the smallest of these cocycles, to reduce the
error-probability.

\begin{theorem} \label{thm:paritycutsv2}
Let $t$ and $c$ be positive integers. 
Let $(G,\tau, \Sigma,\alpha)$ be a feasible connected instance of the $t$-Dimensional Even-Cut Problem and let $k$ be the
optimal value. Then, in $c|V(G)|^4$ repetitions of the Random Contraction Algorithm,
the probability that we fail to find a 
cocycle of $M(G,\tau, \Sigma,\alpha)$ of size $k$ is at most $e^{-24c}$.
\end{theorem}

\begin{proof}
Let $n= |V(G)|$.
By Lemma~\ref{lem:paritycuts1v2}, the error-probability is at most
 \[ \left(1 - \frac{24}{n^4}\right)^{cn^4} \leq \left(e^{-\frac{24}{n^4}}\right)^{cn^4} = e^{-24c}. \qedhere\]
\end{proof}

Combining the connectivity reduction with Theorem~\ref{thm:paritycutsv2}
gives the following result.
\begin{theorem} \label{thm:paritycutsv3}
Let $t$ be a positive integer and let $\epsilon>0$.
There is a randomized algorithm that, given an instance of  the $t$-Dimensional Even-Cut Problem
with $n$ vertices and $m$ edges will, with probability at least $1-\epsilon$,  correctly compute the optimal
value in $O(n^5m)$ time.
\end{theorem}

\section{Perturbations of graphic matroids}

In this section we will see how to ``represent'' perturbations of graphic 
matroids by certain labelled graphs.
Let $s$ and $t$ be non-negative integers.
An {\em $(s,t)$-signed-graft} is a tuple
$(G,S,T,B,C,D)$ such that:
\begin{itemize}
\item $G$ is a graph,
\item $S$ is an $s$-element set disjoint from $V(G)$,
\item $T$ is a $t$-element set disjoint from $E(G)$, 
\item $B\in \GF(2)^{V(G)\times T}$,
\item $C\in \GF(2)^{S\times E(G)}$, and
\item $D\in\GF(2)^{S\times T}$.
\end{itemize}
The {\em incidence matrix} of an $(s,t)$-signed-graft $(G,S,T,B,C,D)$
is the matrix 
$$
A = \bordermatrix{ & E(G) & T \cr  S& C & D \cr V(G) & A(G) & B},
$$
where $A(G)$ is the incidence matrix of $G$.
We denote the matroid $M(A)$ by $M(G,S,T,B,C,D)$.
The following result is well-known but does not, to the best of our knowledge, appear in print.

\begin{lemma}\label{lem:decomposition}
Let $G$ be a graph and let $P\in\GF(2)^{V(G)\times E(G)}$
be a rank-$t$ matrix. Then there is a $(t,t)$-signed-graft
$(G,S,T,B,C,D)$ such that
$$ M(A(G)+P) = M(G,S,T,B,C,D)\con T.$$
\end{lemma}

\begin{proof}
Let $S$ be a $t$-element set disjoint from both $V(G)$ and $E(G)$.
We can find matrices $B\in \GF(2)^{V(G)\times S}$ and $C\in\GF(2)^{S\times E(G)}$ 
such that
$$ P = BC.$$
Now the incidence matrix of the
$(t,t)$-signed-graft $(G,S,S,B,C,-I)$ is 
$$A = \bordermatrix{ & E(G) & S \cr  S& C & -I \cr V(G) & A(G) & B}.$$
Note that $A$ is equivalent, up to row operations, to
$$\bordermatrix{ & E(G) & S \cr  S& C & -I \cr V(G) & A(G)+ P & 0}.$$
Thus $M(A)\con S = M(A(G)+P)$, as required.
\end{proof}

Thus, given an $(s,t)$-signed-graft $(G,S,T,B,C,D)$, we are interested in
computing the girth and the cogirth of $M(G,S,T,B,C,D)\con T$.
For computing cogirth, we can reduce the problem to instances with $s=1$.
\begin{lemma}\label{lem:reduces}
Let $(G,S,T,B,C,D)$ be an $(s,t)$-signed-graft and let 
$S'$ be a one-element set disjoint from $V(G)$.
The cogirth of $M(G,S,T,B,C,D)\con T$ is the minimum of
the cogirths of the matroids
$M(G,S',T,B,yC,yD)\con T$ taken over all
vectors $y\in \GF(2)^{S'\times S}$.
\end{lemma}

\begin{proof}
Note that:
\begin{itemize}
\item each cocyle of $M(G,S',T,B,yC,yD)$ is a cocycle of $M(G,S,T,B,C,D)$, and
\item for each cocycle $C^*$ of $M(G,S,T,B,C,D)$, there exists $y\in \GF(2)^{S'\times S}$,
such that $C^*$ is also a cocycle of $M(G,S',T,B,yC,yD)$.
\end{itemize}
The result now follows easily.
\end{proof}
Given Lemma~\ref{lem:reduces}, we can take the cogirth problem
given by an $(s,t)$-signed-graft and reduce it to $2^s$ 
cogirth problems on $(1,t)$-signed-grafts.
Moreover, given a $(1,t)$-signed-graft $(G,S,T,B,C,D)$,
we can, by Lemma~\ref{formulation1}, formulate
the problem of computing the girth of $M(G,S,T,B,C,D)\con T$
as a $t$-Dimensional Even-Cut Problem. Now, combining
Theorem~\ref{thm:paritycutsv3} with Lemmas~\ref{lem:decomposition}
and~\ref{lem:reduces}, we obtain Theorem~\ref{thm:main2}.

Analogous with Lemma~\ref{lem:reduces},
the girth problem given by an $(s,t)$-signed-graft reduces to $2^t$ 
girth problems on $(s,1)$-signed-grafts.
\begin{lemma}\label{lem:reducet}
Let $(G,S,T,B,C,D)$ be an $(s,t)$-signed-graft and let 
$T'$ be a one-element set disjoint from $E(G)$.
The girth of $M(G,S,T,B,C,D)\con T$ is the minimum of
the girths of the matroids
$M(G,S,T',Bx,C,Dx)\con T'$ taken over all
vectors $x\in \GF(2)^{T\times T'}$.
\end{lemma}

\begin{proof}
If $W$ is a cycle of 
$M(G,S,T,B,C,D)\con T$ then
there exists a cycle $W'$ of $M(G,S,T,B,C,D)$
with $W\subseteq W' \subseteq W\cup T$.
If $x\in \GF(2)^{T\times T'}$ is the characteristic vector
of $W'\cap T$, then 
$W$ is also a cycle of $M(G,S,T',Bx,C,Dx)\con T'$.

Conversely, suppose that $x\in \GF(2)^{T\times T'}$
and that $W$ is a cycle of $M(G,S,T',Bx,C,Dx)\con T'$.
There is a cycle $W'$ of $M(G,S,T',Bx,C,Dx)$ with
$W\subseteq W' \subseteq W\cup T'$.
Note that $|T'|=1$; suppose that $T'=\{v\}$. If $v\not\in W'$ then $W'=W$, so
$W$ is a cycle of $M(G,S,T,B,C,D)$ and, hence,
also of $M(G,S,T,B,C,D)\con T$.
We may therefore suppose that $v\in W'$. Let $X\subseteq T$ such that
$x$ is the characteristic vector of $X$. Now
$W\cup X$ is a cycle of $M(G,S,T,B,C,D)$ and, hence,
$W$ is a cycle of $M(G,S,T,B,C,D)\con T$.
\end{proof}

It remains to find an algorithm that, given an $(s,1)$-signed-graft
$(G,S,T,B,C,D)$, computes the girth of
$M(G,S,T,B,C,D)\con T$.

\section{Perfect matching with parity constraints}

The problem of computing the girth in a perturbed graphic matroid will
be reduced to a minimum $T$-join problem with parity constraints;
that problem, in turn, reduces to a minimum-weight perfect matching problem
with parity constraints. We will solve that matching problem in this section.
Before stating the problem precisely, we need some notation.

Let $G$ be a graph with edge-weights $(w(e)\, :\, e\in E(G))$ in some 
commutative ring (usually $\Z$ or $\GF(2)^t$) and let $M\subseteq E(G)$.
We denote the sum $\sum (w(e)\, : \, e\in M)$ by $w(M)$.

\medskip

{\it 
\noindent
{\bf The $t$-Dimensional Parity Perfect Matching Problem}\\[1mm]
{\sc Instance:} A graph $G$, unary edge-weights  $w : E(G) \rightarrow \Z_{\geq 0}$,
edge-parities $\gamma : E(G) \rightarrow \GF(2)^t$, and a parity demand $\alpha \in \GF(2)^t$.\\
{\sc Problem:} Find a perfect matching $M$ of $G$ minimizing $w(M)$ subject to $\gamma(M) = \alpha$.
}

\medskip

We note that the graphs we are considering may have parallel edges, 
although we may assume that each parallel class contains at most $2^t$ edges, one for each element of $\GF(2)^t$.
Our solution is closely related to the randomized algorithm, of Mulmuley, Vazirani, and Vazirani~\cite{MVV},
for the Exact Matching Problem.

The running time of our algorithm is $\O(w_{\max} n^7 \log^2 n)$ where
$n=|V(G)|$ and $w_{\max}=\max(w(e)\, : \, e\in E(G))$.
Due to the dependence on $w_{max}$ this is not a polynomial-time 
algorithm, but it suffices for our intended application.

\subsection*{The Tutte matrix}
We review the related concepts of Tutte matrices and Pfaffians that we will use to solve the $t$-Dimensional Parity Perfect Matching Problem.
Let $G=(V,E)$ be a graph with $V=\{1,\ldots,n\}$.
Let $x=(x_e\, : \, e\in E)$ be a collection of algebraically independent commuting
indeterminates over $\bR$. The graph $G$ need not be simple; for vertices $u,v\in V$,
we let $E_{uv}$ denote the set of edges with $u$ and $v$ as its ends.
The {\em Tutte matrix} of $G$ is the matrix $T(x)=(t_{uv})_{V\times V}$ where,
for $u,v\in V$,
$$
t_{uv} = \left\{ \begin{array}{ll}
\sum_{e\in E_{uv}} x_e, & u<v \\
-\sum_{e\in E_{uv}} x_e, & u>v \\
0 & u=v.
\end{array}
\right.
$$
Note that $T(x)$ is skew-symmetric; that is, $T(x)$ is equal to the negative of its transpose.
The {\em Pfaffian} of a skew-symmetric matrix $A$, denoted $\Pf(A)$, is a
square-root of its determinant; the Pfaffian of $A$ has an expansion that
is analogous to the permutation expansion of a determinant; see Godsil~\cite{god}.

Let $M$ be a perfect matching of $G$.
We denote the product of $(x_e\, :\, e\in M)$ by $x^M$.
Let $e=u_1u_2$ and $f=v_1v_2$ be edges of $G$ with
$u_1\le u_2$ and $v_1\le v_2$. We say that $e$ and $f$ {\em cross} if
either $u_1<v_1<u_2<v_2$ or $v_1<u_1<v_2<u_2$.
The {\em sign} of $M$, denoted $\sigma_M$, is $(-1)^k$ where $k$ is the number 
of pairs of edges in $M$ that cross. Tutte~\cite{tutte}
observed that
$$
\Pf(T(x)) = \sum_M \sigma_M x^M, 
$$
where the sum is taken over the set of all perfect matchings $M$ of $G$.

The Pfaffian of $T(x)$ is in the ring $\Z[x]$ of polynomials. We will 
extend this by additional indeterminates $z$ and $y=(y_1,\ldots,y_t)$
where $x$, $y$, and $z$ are all algebraically independent and commute.
Now we define the quotient ring
\[ R = \Z[x,y, z] / \langle y_1^2 - 1, \ldots, y_t^2 - 1 \rangle. \]
Since $y_i^2=1$, we will
consider the exponents of $y_i$ as elements of $\GF(2)$.
For $\rho\in \GF(2)^t$, we denote $y_1^{\rho_1}y_2^{\rho_2}\cdots y_t^{\rho_t}$
by $y^{\rho}$.

Let $(G,w,\gamma,\alpha)$ be an instance of $t$-Dimensional Parity Perfect Matching Problem.
We may assume that $V(G)=\{1,\ldots,n\}$. Let $T(x)$ be the Tutte matrix of $G$.
Now we define the {Tutte matrix} of $(G,w,\gamma,\alpha)$ to be the
matrix $T(x,y,z)$ over $R$ obtained from $T(x)$ by replacing each indeterminate
$x_e$ with $x_e y^{\gamma(e)}z^{w(e)}$. Thus
$$
\Pf(T(x,y,z)) = \sum_M \sigma_M x^M y^{\gamma(M)}z^{w(M)}, 
$$
where the sum is taken over the set of all perfect matchings $M$ of $G$.
Now, by collecting like terms we can define a collection of polynomials
$(p_{\beta}(x,z)\, : \, \beta\in \GF(2)^t)$ such that
$$ \Pf(T(x,y,z)) = \sum_{\beta\in \GF(2)^t} p_{\beta}(x,z)y^{\beta}.$$
Given a polynomial $p(x,z)$ in $\Z[x,z]$, we denote the minimum exponent of $z$ among all terms of 
the polynomial $p(x,z)$ by $\mindeg_z(p(x,z))$; if $p(x,z)=0$ the we let $\mindeg_z(p(x,z))=\infty$.
The following lemma follows immediately from the definitions; we omit the proof.
\begin{lemma}\label{summary}
There is a perfect matching $M$ with $\gamma(M) = \alpha$ if and only if
$p_{\alpha}(x,z)\neq 0$. Moreover, if $p_{\alpha}(x,z)\neq 0$, then
$\mindeg_z(p_{\alpha}(x,z))$ is the minimum of $w(M)$
taken over all perfect matchings $M$ with $\gamma(M)=\alpha$.
\end{lemma}

\subsection*{Evaluations}
Let $p$ be a polynomial in $\Z[x_1,\ldots,x_m]$. If $p(x)\neq 0$, then we are unlikely to
get $p(\tilde x)=0$ if we choose $\tilde x\in \Z^{m}$ ``at random".
We start by making this precise. The {\em degree} of $p$ is the maximum,
taken over all terms of $p$, of the sum of the exponents of $x_1,\ldots,x_m$ in the term.
The following result was proved independently by Schwartz~\cite{Schwartz} and Zippel~\cite{Zippel}.

\begin{lemma}[Schwartz-Zippel Lemma] \label{lem:schwartzzippel}
 Let $p$ be a non-zero polynomial in $\Z[x_1,\ldots,x_m]$ with degree at most $d$, and let $S$ be a finite subset of $\Z$. 
 If $\tilde x\in S^{m}$ is chosen randomly, with uniform  probability, then $p(\tilde x) \neq 0$ with 
 probability at least $1 - \frac{d}{|S|}$.
\end{lemma}

Let $n= |V(G)|$ and let $c$ be a positive integer.
We will choose an $\tilde x\in \{1,\ldots,c n\}^{E(G)}$. Note that
$$ \Pf(T(\tilde x,y,z)) = \sum_{\beta\in \GF(2)^t} p_{\beta}(\tilde x,z)y^{\beta}.$$
Using the Schwartz-Zippel Lemma we get the following:
\begin{lemma} \label{getdegree}
Let $c$ be a positive integer. If $\tilde x\in \{1,\ldots,c|V(G)|\}$ is chosen randomly, with uniform  probability, 
then $\mindeg_z(p_{\alpha}(\tilde x,z)) = \mindeg_z(p_{\alpha}(x,z))$ with 
 probability at least $1 - \frac{1}{2c}$.
 \end{lemma}
 
 \begin{proof}
 We may assume that $p_{\alpha}(x,z)\neq 0$ since otherwise the result is trivial.
 Let $k = \mindeg_z(p_{\alpha}(x,z))$. Collect $p_{\alpha}(x,z)$ in like powers of $z$ 
 and let $q(x)$ be the coefficient of $z^k$. Thus $q(x)\neq 0$ and
 $\mindeg_z(p_{\alpha}(\tilde x,z)) = \mindeg_z(p_{\alpha}(x,z))$ if and only if $q(\tilde x)\neq 0$.
 Now $q(x)$ has degree $\frac{1}{2}|V(G)|$, so, by the Schwartz-Zippel Lemma,
 $q(\tilde x)\neq 0$ with probability at least $1-\frac{1}{2c}$, as required.
 \end{proof}

Now the algorithm is obvious.

\medskip
{\it
\noindent
{\bf Random Evaluation Algorithm}\\[1mm]
{\sc Input:} An instance $(G,w,\gamma,\alpha)$ of the  
$t$-Dimension Parity Perfect Matching Problem
and a positive integer $c$.\\[1mm]
{\sc Step 1.} Choose $\tilde x\in\{1,\ldots,c|V(G)|\}^{E(G)}$ uniformly at random.\\
{\sc Step 2.} Compute $\Pf(T(\tilde x,y,z))$ and extract $p_{\alpha}(\tilde x,z)$.\\
{\sc Step 3.} Return $\mindeg_z(p_{\alpha}(\tilde x,z))$.
}
\medskip

By Lemmas~\ref{summary} and~\ref{getdegree},
the algorithm will give the correct answer with 
probability at least $1-\frac{1}{2c}$.

How does one compute $\Pf(T(\tilde x,y,z))$? The easy answer 
is to first compute it over the ring $\Z[y,z]$ before applying the quotient.  
This can be done by applying row and column eliminations;
see Godsil~\cite{god}.
The intermediate matrices would have entries in the 
field of rational functions in $y$ and $z$.
That would result in a running time
of $\O(w_{\max}n^{t+4})$.

Note that $\Pf(T(\tilde x,y,z))$ is in the ring 
$$ \widetilde R = \Z[y,z]/ \langle y_1^2 - 1, \ldots, y_t^2 - 1 \rangle.$$
We can reduce the complexity to $\O(w_{\max} n^7 \log^2 n)$
by computing $\Pf(T(\tilde x,y,z))$ in the ring $\widetilde R$; this is, however, 
not straightforward.

\subsection*{Computing the Pfaffian over a commutative ring}
Mahajan, Subramanya, and Vinay \cite{MahajanSubramanyaVinay} give an algorithm
for computing Pfaffians over arbitrary commutative rings; here we will only focus on
the ring $\widetilde R$. Their algorithm takes polynomially-many ring operations. 
We need to prove that if we use it to
compute $\Pf(T(\tilde x,y,z))$ then each of these ring operations can be done efficiently, 
so that the whole algorithm runs in polynomial time. 

Let $D$ be an $n \times n$ skew-symmetric matrix over the ring $\widetilde R$. Given $D$, Mahajan {\em et al.} 
construct a directed graph $H_D$ with edge weights $\wt(e)$ in $\widetilde R$ and with three 
special vertices $s$, $t^-$, and $t^+$. The construction of
$H_D$ and $\wt$ is given explicitly in the Appendix;
here we will only summarize the properties that are relavent to the subsequent discussion. 
\begin{enumerate}
\item[(H1)] $H_D$ is an acyclic digraph with $2n^3+3$ vertices,
\item[(H2)] each vertex in $H_D$ has in-degree at most $n$, 
\item[(H3)] each directed path in $H_D$ has at most $n+1$ edges, and
\item[(H4)] the weight of each edge is either equal to an entry of $D$ or to one.
 \end{enumerate}

Given a directed path $P$ in $H_D$, we define $\wt^P = \prod_{e \in E(P)} \wt(e)$. 
For vertices $x$ and $y$ of $H_D$, let $\cP(x, y)$ denote the set of all directed $(x, y)$-paths.
They prove the following.

\begin{theorem}[{\cite[Theorem 12]{MahajanSubramanyaVinay}}] \label{thm:pfaffianissumofpaths}
 Let $D$ be a skew-symmetric matrix and $H_D$ the graph described in the Appendix. Then
 \[ \Pf(D) = \sum_{P \in \cP(s, t^+)} \wt^P - \sum_{Q \in \cP(s, t^-)} \wt^Q. \]
\end{theorem}

Using this result, we will show that $\Pf(T(\tilde x,y,z))$ can be computed efficiently.
For an element $r$ of $\widetilde{R}$ we write $\deg_z(r)$ for the degree of $r$ in the variable $z$ and $|r|$ for the largest absolute value of its integer coefficients.
We assume that multiplying two integers $a$ and $b$ takes time $\O(\log a \log b)$ and that adding them takes time $\O(\max\{\log a, \log b\})$.

\begin{lemma} \label{lem:timetocomputepfaffian}
Let $t$ be a fixed non-negative integer.
Now let $D = (d_{ij})$ be an $n \times n$ matrix over $\widetilde{R}$, $k = \max\{\deg_z(d_{ij})\, :\,1 \leq i,j \leq n\}$, and $c = \max\{|d_{ij}|\, : \,1 \leq i,j \leq n\}$. If each entry of $D$ has at most $2^t$ terms and $c \leq n$, then we can compute $\Pf(D)$ in $\O(n^7 k \log^2 n)$ time.
\end{lemma}

\begin{proof}
Let $H_D$ be the directed graph associated with $D$ described in the Appendix.
For each node $v$ of $H_D$, let $f(v) = \sum_{P \in \cP(s, v)} \wt^P$. 
By Theorem~\ref{thm:pfaffianissumofpaths} we need only compute $f(t^+)$ and $f(t^-)$.
Since $H_D$ is acycilc, there is an ordering $(v_1,v_2,\ldots,v_{2n^3+3})$ such that
for each edge $e=v_iv_j$ of $H_D$ we have $i<j$.
We denote by $\delta^-(v)$ the set of vertices that are tails of arcs with head $v$.
Note that $f(v) = \sum_{u \in \delta^-(v)} f(u)\wt(uv)$ and for each $v_i\in \delta^-(v_j)$  we have $i<j$.
So this formula allows us to compute $f(v_1),\, f(v_2),\ldots,f(v_{2n^3+3})$ in that order.

By property (H4), for each edge $e$ we have $|\wt(e)| \leq c$ and
$\deg_z(\wt(e)) \leq k$. Then, by property (H3),
for each $v \in V(H_D)$, we have $\deg_z(f(v)) \leq (n+1)k$.
Moreover, by properties (H2) and (H3) we have 
$|\cP(s,v)|\le (n-1)^n$ and, hence, $|f(v)| \leq c^{n+1}(n+1)^{n}$.

To compute each $f(v)$, we first do $\O(n)$ multiplications of the form $f(u) \wt(uv)$. Since each $\wt(uv)$ has a constant number of terms, each multiplication takes time $\O(\deg_z(f_u) \log|f_u| \log|\wt(uv)|)$ or $\O(n^2 k \log(cn)\log c)$.
We then do $\O(n)$ additions of these terms, each of which takes time $\O(nk \log(c^{n+1}(n+1)^{n}))$ or $\O(n^2 k \log(cn))$. 
So computing each $f(v)$ takes time $\O(n^3 k \log(cn)\log c)$.

Since there are $\O(n^3)$ nodes in $H_D$ and $c \leq n$, the total time taken is $\O(n^6 k \log^2 n)$.
\end{proof}

Given an instance $(G,w,\gamma,\alpha)$ of $t$-Dimensional Parity Perfect Matching Problem, we may assume that,
for each $\beta\in\GF(2)^t$, there is at most one edge $e$ with $\gamma(e)=\beta$ between
any two given vertices, since otherwise we could delete all but the one of minimum weight.
This gives the following result.
\begin{theorem}\label{thm:matching}
Let $t$ be a non-negative integer and $\epsilon>0$.
There is a randomized algorithm that, given an instance of the 
$t$-Dimensional Parity Perfect Matching Problem
with $n$ vertices and maximum edge weight $w_{\max}$,
correctly solves the problem, with probability at least $1-\epsilon$,
in time $\O(w_{\max} n^6 \log^2 n)$.
\end{theorem}

\section{Walks, cycles, and joins with parity constraints} 

Let $G$ be a graph, let $\gamma\in\GF(2)^{E(G)}$, and let $u,v\in V(G)$.
A {\em $(u,v)$-walk} in $G$ is a sequence
$(v_0,e_1,v_1,e_2,v_2,\ldots,v_{k-1},e_k,v_k)$ such that
\begin{itemize}
\item $v_0,v_1,\ldots,v_k$ are vertices with $v_0=u$ and $v_k=v$, and
\item for each $i\in\{1,\ldots,k\}$, the element $e_i$ is an edge with ends $\{v_{i-1},\, v_i\}$.
\end{itemize}
If $W=(v_0,e_1,v_1,e_2,v_2,\ldots,v_{k-1},e_k,v_k)$ is a walk, then the {\em length}
of $W$ is $k$ and the {\em parity} of $W$, denoted $\gamma(W)$, is
$\gamma(e_1)+\cdots+\gamma(e_k)$.
We denote by $E_{\odd}(W)$ the set of edges that occur an odd number 
of times in $W$; note that $\gamma(E_{\odd}(W))= \gamma(W)$ and $|E_{\odd}(W)|\le k$.

The first problem that we solve in this section is:

\medskip

{\it
\noindent
{\bf The $t$-Dimensional Parity Walk Problem}\nopagebreak \\[1mm]
{\sc Instance:}
A tuple $(G, \gamma, \alpha,u,v)$
where $G$ is a graph and $u$ and $v$ are vertices of $G$,
$\gamma: E(G)\rightarrow\GF(2)^{t}$, and $\alpha\in\GF(2)^t$.\\
{\sc Problem:}
Find a minimum length $(u,v)$-walk $W$ in $G$ of parity $\alpha$.
}
\medskip

We define a simple graph $G^\gamma$ as follows.
The vertex set of $G^\gamma$ is $V(G)\times \GF(2)^t$.
Vertices $(u,\beta)$ and $(v,\beta')$ are adjacent if and only if 
there is an edge $e=uv$ of $G$ with $\gamma(e)=\beta+\beta'$.
Note that $|V(G^\gamma)|= 2^t|V(G)|$ and $|E(G^\gamma)|\le 2^t|E(G)|$
(with equality unless there are two edges in $G$ with both the same ends and the same parity).
The following result is an easy consequence of this construction.
\begin{lemma}\label{lem:walks}
Let $(G, \gamma, \alpha,u,v)$ be an instance of the $t$-Dimensional Parity Walk Problem.
Then there is a $(u,v)$-walk in $G$ of length $k$ and parity $\alpha$
if and only if there is a $((u,0),(v,\alpha))$-walk in $G^\gamma$ of length $k$.
\end{lemma}

Now, by Lemma~\ref{lem:walks}, the $t$-Dimensional Parity Walk Problem
can be solved in linear time.

\subsection*{Cycles with parity constraints}
Next we consider the following problem.
\medskip

{\it
\noindent
{\bf The $t$-Dimensional Parity Cycle Problem}\nopagebreak \\[1mm]
{\sc Instance:}
A tuple $(G, \gamma, \alpha)$
where $G$ is a graph,
$\gamma: E(G)\rightarrow\GF(2)^{t}$, and $\alpha\in\GF(2)^t$.\\
{\sc Problem:}
Find a minimum cardinality cycle $C$ of $G$ with $\gamma(C)=\alpha$.
}
\medskip

Note that we do not require $C$ to be non-empty.

Let $(G, \gamma, \alpha)$ be an instance of the $t$-Dimensional Parity Cycle Problem.
A {\em closed walk} in a graph is a $(u,u)$-walk for any vertex $u$.
For each $\beta\in \GF(2)^t$, let $w(\beta)$ be the minimum length of a
closed walk of parity $\beta$; if there is no such walk we define
$w(\beta)=\infty$. Note that, if $G$ has $n$ vertices and $m$ edges, then
we can compute $w(\beta)$ in $\O(nm)$ time.
Now we define $\tilde w(\beta)$ to be the minimum of
$\sum(w(\alpha)\, : \, \alpha\in S)$ taken over all subsets
$S\subseteq \GF(2)^t$ with $\sum(\alpha\, :\, \alpha\in S) =\beta$.
Since we treat $t$ as constant, we can compute $\tilde w(\beta)$ in
$\O(nm)$ time.

\begin{lemma}\label{lem:cycles}
Let $(G, \gamma, \alpha)$ be an instance of the $t$-Dimensional Parity Cycle Problem.
Then the optimal value is $\tilde w(\alpha)$.
\end{lemma}

\begin{proof}
Let $C$ be a smallest cycle with $\gamma(C)=\alpha$.
Since the symmetric difference of two cycles is a cycle, $|C|\le\tilde w(\alpha)$.
On the other hand, consider a partition  $(C_1,\ldots,C_k)$ of $C$ into circuits.
For each $i\in \{1,\ldots,k\}$, let $\beta_i=\gamma(C_i)$;
note that $|C_i|\ge w(\beta_i)$.
So $|C|\ge w(\{\beta_1,\ldots,\beta_k\})\ge \tilde w(\alpha)$.
\end{proof}

By Lemma~\ref{lem:cycles}, we can solve the $t$-Dimensional Parity Cycle Problem
in $\O(nm)$ time.

\subsection*{Joins with parity constraints}
Let $G=(V,E)$ be a graph and let $T\subseteq V$. A {\em $T$-join} is a set $J\subseteq E$
such that $T$ is the set of odd-degree vertices of the subgraph $G[V,J]$.
Since graphs have an even number of vertices of odd degree, there do not exist
$T$-joins unless $|T|$ is even. Now consider the following problem.

\medskip

{\it
\noindent
{\bf The $t$-Dimensional Parity Join Problem}\nopagebreak \\[1mm]
{\sc Instance:}
A tuple $(G, T, \gamma, \alpha)$
where $G$ is a graph, $T\subseteq V(G)$,
$\gamma: E(G)\rightarrow\GF(2)^{t}$, and $\alpha\in\GF(2)^t$.\\
{\sc Problem:}
Find a minimum size $T$-join $J$ of $G$ with $\gamma(J)=\alpha$.
}
\medskip

Let $G$ be a graph, $T\subseteq V(G)$, and
$\gamma: E(G)\rightarrow\GF(2)^{t}$. For $\beta\in \GF(2)^t$,
we let $\tilde w_T(\beta)$ denote the 
minimum size of a $T$-join $J$ of $G$ with $\gamma(J)=\beta$;
if there is no $T$-join $J$ with $\gamma(J)=\beta$, then we let
$\tilde w_T(\beta)=\infty$.

Note that an $\emptyset$-join is a cycle, so $\tilde w_\emptyset(\beta) = 
\tilde w(\beta)$. Thus, when $T=\emptyset$, the 
$t$-Dimensional Parity Join Problem is just the
$t$-Dimensional Parity Cycle Problem, which is solved above.

We will start by considering instances $(G, T, \gamma, \alpha)$ with $|T|=2$.
For $u,v\in V(G)$ and $\beta\in \GF(2)^t$, we let $w_{uv}(\beta)$ denote
the minimum length of a $(u,v)$-walk of parity $\beta$;
we let $w_{uv}(\beta)=\infty$ if no such walk exists.
\begin{lemma}\label{lem:twoelements}
Let $(G, T,  \gamma, \alpha)$ be an instance of the $t$-Dimensional Parity Join Problem
with $T=\{u,v\}$.
Then $\tilde w_T(\alpha)$ is equal to the minimum of $w_{uv}(\beta) + \tilde w(\alpha+\beta)$
taken over all $\beta\in \GF(2)^t$. 
\end{lemma}

\begin{proof}
If $W$ is a $(u,v)$-walk with $\gamma(W)=\beta$ and $C$ is a cycle with 
$\gamma(C) =\alpha+\beta$ then $E_{\odd}(W)\Delta C$ is a $\{u,v\}$-join
with $\gamma(E_{\odd}(W)\Delta C)= \alpha$. Thus
$$\tilde w_T(\alpha) \le w_{uv}(\beta) + \tilde w(\alpha+\beta).$$

Conversely, consider a $\{u,v\}$-join $J$ with $\gamma(J)=\alpha$ and
$|J| = \tilde w_T(\alpha)$. Since $u$ and $v$ are the only
odd-degree vertices in $G[V,J]$, there is a $(u,v)$-path in  $G[V,J]$;
let $P$ be the edge set of such a path and let $\beta=\gamma(P)$.
Now $|P|\ge w_{uv}(\beta)$ and, since $J-P$ is a cycle with $\gamma(J-P)=\alpha+\beta$,
we have $|J-P|\ge \tilde w(\alpha+\beta)$. Thus
$$\tilde w_T(\alpha) \ge w_{uv}(\beta) + \tilde w(\alpha+\beta);$$
which completes the proof.
\end{proof}

Then, as an immediate corollary, we get the following result.
\begin{theorem}\label{thm:twoelements}
There is an algorithm that,
given an instance $(G, T,  \gamma, \alpha)$ of the $t$-Dimensional Parity Join Problem
with $n$ vertices and $m$ edges and with $|T|\le 2$, finds the optimal value
in time $\O(mn)$.
\end{theorem}

We also need the following bound on $\tilde w_{\{u,v\}}(\beta)$.
\begin{lemma}\label{lem:size}
Let $(G, T,  \gamma, \alpha)$ be an instance of the $t$-Dimensional Parity Join Problem
with $|T|= 2$. If $\tilde w_T(\alpha)$ is finite, then $\tilde w_T(\alpha)\le 2^t |V(G)|$.
\end{lemma}

\begin{proof}
Let $T=\{u,v\}$.
Let $J$ be a $T$-join with $\gamma(J) =\alpha$.
We can partition $J$ into sets $(P,C_1,\ldots,C_k)$
where $P$ is the edge set of a $(u,v)$-path and 
$C_1,\ldots,C_k$ are the edge sets of circuits.
We may assume that:
\begin{itemize}
\item for each $i\in\{1,\ldots,k\}$, we have $\gamma(C_i)\neq 0$ (since otherwise
$J-C_i$ is a $T$-join with $\gamma(J-C_i)=\alpha$), and
\item for each $1\le i<j\le k$,
we have $\gamma(C_i)\neq \gamma(C_j)$ (since otherwise
$J-C_i-C_j$ is a $T$-join with $\gamma(J-C_i-C_j)=\alpha$).
\end{itemize}
Thus $k\le 2^t-1$ and hence $|J|\le 2^t|V(G)|$.
\end{proof}

We now return to the general case.
Let $(G, T,  \gamma, \alpha)$ be an instance of the $t$-Dimensional Parity Join Problem.
Let $T^2$ denote the set of $2$-element subsets of $T$.
We define a graph $H(G,T)$ with vertex set $T$ and edge set
$T^2\times \GF(2)^t$ where the edge $(\{u,v\},\beta)$ has ends $u$ and $v$.
Given an edge $e=(\{u,v\},\beta)$ we define $\tilde w(e) = \tilde w_{\{u,v\}}(\beta)$
and $\tilde\gamma(e) =\beta$.
We can construct $H(G,T)$ and all of its associated ``edge weights'', $\tilde w(e)$
and $\tilde\gamma(e)$, in $\O(n^3m)$ time. 
\begin{lemma}\label{lem:joins}
Let $(G, T,  \gamma, \alpha)$ be an instance of the $t$-Dimensional Parity Join Problem
with $|T| > 0$ even.
Then $\tilde w_T(\alpha)$ is equal to the minimum of $\tilde w(M)$ 
taken over all perfect matchings $M$  of $H(G,T)$ with $\tilde \gamma(M)=\alpha$.
\end{lemma}

\begin{proof}
Consider a $T$-join $J$ with $\gamma(J)=\alpha$ and $|J| = \tilde w_T(\alpha)$.
Let $k=\frac 1 2 |T|$.

\begin{claim}
There is a partition $(T_1,\ldots,T_k)$ of $T$ into two-element sets
and a partition $(J_1,\ldots,J_k)$ of $J$ such that, for each $i\in\{1,\ldots,k\}$,
the set $J_i$ is a $T_i$-join.
\end{claim}

\begin{proof}[Proof of claim]
The proof is by induction on $k$, and is trivial when $k = 1$; suppose that $k\ge 2$.
Let $u\in T$. The component of $G[V,J]$ that contains $u$ must 
contain another vertex of odd degree, say $v$. Since
$J$ is a $T$-join, $v\in T$. Let $J_1$ be the edge set of a $(u,v)$-path in 
$G[V,J]$ and let $T_1=\{u,v\}$. Thus $J_1$ is a $T_1$-join and $J-J_1$ is a $(T-T_1)$-join.
So the result follows by induction. 
\end{proof}

Let $(T_1,\ldots,T_k)$ and $(J_1,\ldots,J_k)$ be as in the claim.
Now, for each $i\in \{1,\ldots, k\}$, let $e_i =(T_i,\gamma(J_i))$; note that
$e_i$ is an edge of $H(G,T)$ and $|J_i|\ge \tilde w(e_i)$.
Let $M=\{e_1,\ldots,e_k\}$. Now $M$ is a perfect matching of $H(G,T)$
with $\tilde \gamma(M)=\gamma(J) = \alpha$ and
$\tilde w(M) \le |J| = \tilde w_T(\alpha)$.

Conversely, consider a matching $M$ of $H(G,T)$ with $\tilde \gamma(M)=\alpha$.
For each edge $e=(T',\beta)\in M$, let $J_e$ be a $T'$-join
with $\gamma(J_e)=\beta$ and $|J_e|= \tilde w(e)$.
Now let $J$ be the symmetric difference of the sets $(J_e\, : \, e\in M)$.
Then $J$ is a $T$-join, $\gamma(J) = \alpha$, and $|J|\le \tilde w(M)$. 
Hence $\tilde w_T(\alpha) \le \tilde w(M)$, as required.
\end{proof}

Lemma~\ref{lem:joins} enables us to reduce an instance 
$(G, T,  \gamma, \alpha)$ of the $t$-Dimensional Parity Join Problem
to an instance $(H(G,T),\tilde w,\tilde \gamma,\alpha)$ of 
the $t$-Dimensional Parity Perfect Matching Problem.
We will delete the edges $e$ of $H(G,T)$ with $\tilde w(e) = \infty$.
Then, by Lemma~\ref{lem:size}, $\tilde w_{\max}=\O(n)$.
So, by Theorem~\ref{thm:matching}, we get the following result.
\begin{theorem}\label{thm:joins}
Let $t$ be a non-negative integer and $\epsilon>0$.
There is a randomized algorithm that, given an instance of the 
$t$-Dimensional Parity Join Problem
with $n$ vertices, correctly solves the problem, with probability at least $1-\epsilon$,
in time $\O(n^7 \log^2 n)$.
\end{theorem}

\section{Computing girth}

This brings us to the final step.
\begin{theorem}\label{coupdegrace}
Let $t$ be a non-negative integer and $\epsilon>0$.
There is a randomized algorithm that, given a $(t,1)$-signed-graft
$(G,S,T,B,C,D)$ with $n$ vertices and $m$ edges, correctly 
computes the girth of the matroid $M(G,S,T,B,C,D)\con T$, with probability at least $1-\epsilon$,
in time $O(n^7 \log^2 n+ mn)$.
\end{theorem}

\begin{proof}
In $\O(mn)$ time we can check whether
$M(G,S,T,B,C,D) / T$ has a loop. If not, but $m > 2^t\binom{n}{2}$, then its girth is two, so
we may assume that $m = \O(n^2)$. In $\O(n^5)$ time we can check if $M(G,S,T,B,C,D) / T$ has any parallel pairs, so
we may assume that it is simple.

Suppose that $T=\{e\}$.
Let $k_1$ be the girth of $M(G,S,T,B,C,D)\del T$ and let 
$k_2+1$ be the size of the smallest cycle of $M(G,S,T,B,C,D)$ that 
contains $e$. If $k_2=0$, then the girth of  $M(G,S,T,B,C,D)\con T$
is $k_1$ and, if $k_2>0$, then the girth of $M(G,S,T,B,C,D)\con T$
is $\min(k_1,k_2)$. In either case, it suffices to compute $k_1$ and $k_2$.

We may assume that $S=\{1,\ldots,t\}$. For each edge $f$ of $G$
we let $\gamma(f)$ denote the column of $C$ indexed by $f$.
Let $\alpha$ be the unique column of $D$, and 
let $\widetilde T\subseteq V(G)$ be the set whose characteristic vector is 
the unique column of $B$.

\begin{claim} 
We can compute $k_1$ in $\O(m^2 n)$ time.
\end{claim}

\begin{proof}[Proof of Claim]
Let $f$ be an edge of $G$, and let $\beta = \gamma(f)$. If $f$ is not a loop,
then let $T'$ be the set of ends of $e$. If $f$ is a loop, 
let $T'=\emptyset$. For $J\subseteq E(G)-\{f\}$, the set
$J\cup \{f\}$ is a cycle of $M(G,S,T,B,C,D)\del T$ if and only if
$J$ is a $T'$-join in $G-f$ with $\gamma(J)=\beta$.
So, by Lemma~\ref{thm:twoelements}, we can 
compute the size of the shortest cycle in $M(G,S,T,B,C,D)\del T$
that contains $f$ in $\O(mn)$ time. Thus we can compute
the girth of $M(G,S,T,B,C,D)\del T$ by repeating this for each edge of $G$.
\end{proof}

For $J\subseteq E(G)$, the set $J\cup \{e\}$ is a cycle of
$M(G,S,T,B,C,D)$ if and only if $J$ is a $\widetilde T$-join
in $G$ with $\gamma(J)=\alpha$. Then,
by Theorem~\ref{thm:joins}, there is a randomized algorithm
that will correctly compute $k_2$, with probability at least $1-\epsilon$,
in $\O(n^7\log^2 n)$ time.
\end{proof}

Note that Theorem~\ref{thm:main1} follows from Lemmas~\ref{lem:decomposition} and~\ref{lem:reducet}
and Theorem~\ref{coupdegrace}.

\section*{Acknowledgement}
We thank Bert Gerards and Geoff Whittle, with whom we have enjoyed many fruitful discussions on this topic. We also thank the referee for their suggestions.

\appendix

\section{The construction of $H_D$}

Let $D=(d_{ij})$ be an $n$ by $n$ skew-symmetric matrix with entries
in a ring $R$. We may assume that the rows and columns are indexed
by $\{1,\ldots,n\}$. If $n$ is odd, the the Pfaffian is zero, so we may 
assume that $n$ is even. Let
$$ X=\{0,1\}\times\{1,\ldots,n\}\times\{1,\ldots,n\}\times\{0,\ldots,n-1\}.$$
Now let $H_D = (V,E)$ be the edge-weighted directed 
graph where $ V=\{s,t^-,t^+\}\cup X$ 
and $E$ consists of the following edges:
\begin{enumerate}
\item For each $a\in \{1,3,\ldots,n-1\}$, there is an edge
$e=sv$ where $v=(0,a,a,0)$ with $\wt(e)=1$.
\item For each $v\in X$  with $v_4\in\{2,4,\ldots,n\}$ and $v_3>v_2$, and for each $a\in\{1,\ldots, n\}-\{v_3\}$,
if $a<v_3$ there 
is an edge $e=uv$ where $u=(1-v_1, v_2,a,v_4-1)$ with $\wt(e) = d_{u_4,v_4},$ and
if $a>v_3$ there
is an edge $e=uv$ where $u=(v_1,v_2,a,v_4-1)$ with $\wt(e) = d_{v_4,u_4}.$
\item For each $v\in X$ with $v_3$ and $v_4$ both even and with $v_2< v_3-1$, there 
is an edge $e=uv$ where $u=(1-v_1,v_2,v_3-1,v_4-1)$ with $\wt(e) = 1.$
\item For each $v\in X$ with $v_3$ odd, $v_4$ even and with $v_2< v_3$, there
$u_1=v_1$, $u_2=v_2$, $u_3=v_3+1$, and $u_4=v_4-1$, there 
is an edge $e=uv$ where $u=(v_1,v_2,v_3+1,v_4-1)$ with $\wt(e) = 1.$
\item For each $v\in X$ with $v_3$ odd, $v_4$ even, and $v_3=v_2$, and for each $a\in\{1,\ldots, v_2-1\}$ there
is an edge $e=uv$ where $u=(1-v_1, a, a+1, v_4-1)$ with $\wt(e) = 1.$
\item For each $a\in\{1,3,\ldots,n-1\}$ there
is an edge $e=ut^{-}$ where $u=(0,a,a+1,n-1)$ with $\wt(e) = 1.$
\item For each $a\in\{1,3,\ldots n-1\}$ there
is an edge $e=ut^{+}$ where $u=(1,a,a+1,n-1)$ with $\wt(e) = 1.$
\end{enumerate}
 
\begin{lemma} The directed graph $H_D$ satisfies the following:
\begin{itemize}
\item[(H1)] $H_D$ is an acyclic digraph with $2n^3+3$ vertices
\item[(H2)] each vertex in $H_D$ has in-degree at most $n$, 
\item[(H3)] each directed path in $H_D$ has at most $n+1$ edges, and
\item[(H4)] the weight of each edge is either equal to an entry of $D$ or to one.
\end{itemize}
\end{lemma}

\begin{proof} Note that, by definition, 
$H_D$ has $2n+3$ vertices and (H4) holds.
Let $X_0=\{s\}$, $X_{n_1}=\{t^-,t^+\}$, and, for each $i\in\{0,\ldots,n\}$,
let
$$ X_{i} =\{0,1\}\times\{1,\ldots,n\}\times\{1,\ldots,n\}\times\{i-1\}.$$
For each edge $e=uv$ there exists $i\in \{0,\ldots,n\}$ such that
$u\in X_i$ and $v\in X_{i+1}$. It follows that
$H_D$ is acyclic and has no directed path with more than $n+1$ edges.
 Thus (H1) and (H3) hold. 
 
Consider a vertex $v$ of $H_D$. Each edge $e=uv$ entering
$v$ is of one of the types $(1),\ldots,(7)$. Note that there 
are at most $n$ edges of any given type entering $v$ and that
there is at most one type of edge entering $v$. For example,
if there is an edge $e=uv$ of type (2) entering $v$ then
$v\in X$ and  $v_4\in\{2,4,\ldots,n\}$, which precludes the existence of other types
of edges entering $v$. Thus $v$ has in-degree at most $n$.
\end{proof}

\end{document}